\documentclass[12pt,leqno]{amsart}
\usepackage{pifont}
\usepackage{stmaryrd}
\usepackage{dsfont}
\usepackage{color}
\usepackage{mathrsfs}
\usepackage{amsmath}
\usepackage{amssymb}
\usepackage{hyperref}
\usepackage{bookmark}
\usepackage[bottom]{footmisc}
\usepackage{verbatim}
\usepackage{extarrows}%提供\xlongequal命令
\usepackage{mathtools}%提供双向箭头上下可写字

\numberwithin{equation}{section}
\newtheorem{thm}{Theorem}[section]
\newtheorem{lem}[thm]{Lemma}

\newtheorem{prop}[thm]{Proposition}             

\newtheorem{defi}[thm]{Definition}
\newtheorem{rmk}[thm]{Remark}

\newcommand{\wt}{\operatorname{wt}}
\newcommand{\xqz}[1]{\lfloor #1 \rfloor}
\newcommand{\Res}{\operatorname{Res}}

\renewcommand{\mod}{\operatorname{mod}}

\def\Z{\mathbb{Z}}
\def\N{\mathbb{N}}

\def\D{\mathcal{D}}

\allowdisplaybreaks

\begin{document}

\title{Twisted associative algebras associated to vertex algebras}

\author{Shun Xu}

\address{School of Mathematical Sciences, Tongji University, Shanghai, 200092, China}

\email{shunxu@tongji.edu.cn}



\begin{abstract}
Let $V$ be a vertex  algebra and $g$ an automorphism of $V$ of
order $T$. We construct a sequence of
associative algebras  $\tilde{A}_{g,n}(V )$ for any $n\in(1/T)\N$,
which are not depend on the conformal structure of $V$. We  show that for a vertex operator algebra, $g$-rationality, $g$-regularity,
and twisted fusion rules are independent of the choice of the conformal
vector.
\end{abstract}

\maketitle

\section{Introduction}

For a vertex operator algebra $V$, Zhu \cite{Z1} constructed an associative algebra $A(V)$ and established a one-to-one correspondence between the isomorphism classes of irreducible $A(V)$-modules and those of irreducible admissible $V$-modules.
Later, Frenkel and Zhu \cite{FZ1} constructed an $A(V)$-$A(V)$-bimodule for admissible $V$-modules. Li \cite{L1} used this bimodule to establish a theorem on fusion rules. Subsequently, Y. Zhu \cite{Z2} generalized this result to the twisted case of fusion rules.
Following the idea that one can construct associative algebras from $V$, Dong, Li, and Mason \cite{DLM1} introduced a family of associative algebras $A_n(V)$ for all $n \in \mathbb{N}$, where $A_0(V)=A(V)$, and proved that $V$ is rational if and only if each $A_n(V)$ is finite-dimensional and semisimple.
Later, Dong and Jiang \cite{DJ1} constructed $A_n(V)$-$A_m(V)$-bimodules $A_{n,m}(V)$ for arbitrary $n, m \in \mathbb{N}$, where $A_{n,n}(V)=A_n(V)$, using which they gave a more precise construction of Verma-type admissible $V$-modules. Specifically, for any $A_m(V)$-module $U$, the direct sum
$
\bigoplus_{n \in \mathbb{N}} A_{n,m}(V) \otimes_{A_m(V)} U
$
is a Verma-type admissible $V$-module. It should be emphasized that $A_{n,m}(V)$ admits a definition in terms of representation theory \cite{HJZ1}, and can be realized as a specific subquotient of the universal enveloping algebra $U(V)$ associated with $V$ (see \cite{HX1,HJZ1}).

For vertex operator algebras $(V, \omega)$ and $(V, \omega')$, where $\omega$ and $\omega'$ are both conformal vectors of $V$, a natural question is whether the rationality of $V$ depends on the choice of its conformal vector. Li \cite{L2} constructed a family of associative algebras $\tilde{A}_n(V)$ for all $n \in \mathbb{N}$, associated to the vertex algebra $V$, whose construction does not depend on the choice of conformal vector. If $V$ is a vertex operator algebra and $\omega$ is a conformal vector of $V$, Li proved that $\tilde{A}_n(V)$ is isomorphic to $A_n(V)$ for all $n \in \mathbb{N}$, thereby giving an affirmative answer to the above question. Li also show that for a vertex operator algebra, 
regularity and  fusion rules are independent of the choice of the conformal vector.
H. Li and Wang \cite{LW1} extended this idea to the setting of $\mathbb{Z}$-graded vertex operator superalgebras. In particular, they constructed a family of associative superalgebras $\tilde{A}_n(V)$ indexed by nonnegative integers $n$, associated to a given vertex superalgebra $V$, and these algebras are independent of the choice of conformal structure on $V$. Based
on this, they show that the rationality, regularity and fusion
rules are independent of the choice of the conformal vector
for $\Z$-graded vertex operator superalgebra.

We now consider the twisted representation theory of vertex operator algebras. Let $V$ be a vertex operator algebra and let $g$ be an automorphism of $V$ of finite order $T$. Dong, Li, and Mason \cite{DLM2} constructed an associative algebra $A_g(V)$, where $A_{\operatorname{id}_V}(V) = A(V)$, and established a one-to-one correspondence between the isomorphism classes of irreducible $A_g(V)$-modules and those of irreducible admissible $g$-twisted $V$-modules.
For each $n \in ({1}/{T})\mathbb{N}$, they further introduced a family of associative algebras $A_{g,n}(V)$ \cite{DLM3} such that $A_{g,0}(V) = A_g(V)$ and $A_{\operatorname{id}_V,n}(V) = A_n(V)$. They proved that $V$ is $g$-rational if and only if each $A_{g,n}(V)$ is finite-dimensional and semisimple.
Later, Dong and Jiang \cite{DJ2} constructed $A_{g,n}(V)$-$A_{g,m}(V)$-bimodules $A_{g,n,m}(V)$ for arbitrary $n, m \in ({1}/{T})\mathbb{N}$, with $A_{g,n,n}(V) = A_{g,n}(V)$. Using these bimodules, they provided a more refined construction of Verma-type admissible $g$-twisted $V$-modules. Specifically, for any $A_{g,m}(V)$-module $U$, the direct sum
$
\bigoplus_{n \in ({1}/{T})\mathbb{N}} A_{g,n,m}(V) \otimes_{A_{g,m}(V)} U
$
is a Verma-type admissible $g$-twisted $V$-module.
It should be emphasized that $A_{g,n,m}(V)$ admits a definition from the perspective of representation theory \cite{HXX1} and can be realized as a specific subquotient of the universal enveloping algebra $U(V[g])$ associated to $V$ with respect to the automorphism $g$ (see \cite{HXX1,HX2}).

For vertex operator algebras $(V, \omega)$ and $(V, \omega')$, let $g$ be an automorphism of finite order $T$ of both structures. Naturally, one may ask whether the $g$-rationality of $V$ depends on the choice of its conformal vector. In the case $g = \operatorname{id}_V$, this question was already answered by Li in \cite{L2}.
In this paper, for a vertex  algebra $V$ and  an automorphism $g$ of $V$ of
order $T$, we construct a sequence of associative algebras $\tilde{A}_{g,n}(V)$ for all $n \in ({1}/{T})\mathbb{N}$, which are independent of the conformal structure of $V$. When $V$ is vertex operator algebra, we show that each $\tilde{A}_{g,n}(V)$ is isomorphic to the algebra $A_{g,n}(V)$ constructed by Dong, Li, and Mason. Furthermore, we show that for a vertex operator algebra, properties such as $g$-rationality, $g$-regularity, and twisted fusion rules do not depend on the choice of the conformal vector.

This paper is organized as follows. In Section 2, we review some basic concepts and preliminary results from the twisted representation theory of vertex operator algebras that will be needed in the sequel.
In Section 3, for a vertex algebra $V$ and an automorphism $g$ of $V$ of finite order $T$, we construct a sequence of associative algebras $\tilde{A}_{g,n}(V)$ for all $n \in ({1}/{T})\mathbb{N}$, which are independent of the conformal structure of $V$.
In Section 4, we show that for a vertex operator algebra, the properties of $g$-rationality, $g$-regularity, and twisted fusion rules do not depend on the choice of the conformal vector.
\section{Preliminaries}
In this section, we review some basic concepts and preliminary results from the twisted representation theory of vertex operator algebras that will be needed in the sequel.
\begin{defi}
A vertex algebra consists of a vector space $V$ equipped with a linear map
\begin{align*}
Y(\cdot, x) \colon V &\to (\operatorname{End} V)[[x, x^{-1}]] \\
v &\mapsto Y(v, x) = \sum_{n \in \mathbb{Z}} v_n x^{-n-1}
\end{align*}
and a distinguished vector $\mathbf{1} \in V$, called the vacuum vector, satisfying the following conditions for any $u, v \in V$:

\begin{enumerate}
    \item $Y(\mathbf{1}, z) = \operatorname{id}_V$, and $u_n \mathbf{1} = \delta_{n,-1} u$ for all $n \geq -1$;
    
    \item $u_n v = 0$ for sufficiently large $n$;
    
    \item the Jacobi identity:
    $$
    \begin{gathered}
    z_0^{-1} \delta\left(\frac{z_1 - z_2}{z_0}\right) Y(u, z_1) Y(v, z_2)
    - z_0^{-1} \delta\left(\frac{z_2 - z_1}{-z_0}\right) Y(v, z_2) Y(u, z_1) \\
    = z_2^{-1} \delta\left(\frac{z_1 - z_0}{z_2}\right) Y(Y(u, z_0)v, z_2).
    \end{gathered}
    $$
\end{enumerate}
\end{defi}

Let $V$ be a vertex algebra. Define a linear operator $\mathcal{D}$ on $V$ by
$$
\mathcal{D}(v) = v_{-2} \mathbf{1} = \lim_{x \to 0} \frac{d}{dx} Y(v, x)\mathbf{1}, \quad \text{for } v \in V.
$$
It follows from \cite{LL1} that for any $u, v \in V$,
$$
[\mathcal{D}, Y(v, x)] = Y(\mathcal{D}v, x) = \frac{d}{dx} Y(v, x),
$$
and
\begin{equation}\label{eq2.1}
Y(u, x)v = e^{x\mathcal{D}} Y(v, -x)u.
\end{equation}

\begin{defi}
Let $(V, Y, \mathbf{1})$ be a vertex algebra. A linear isomorphism $g$ of $V$ is called an automorphism of $V$ if it satisfies
$$
g(\mathbf{1}) = \mathbf{1} \quad \text{and} \quad g(Y(u, z)v) = Y(g(u), z)g(v) \quad \text{for all } u, v \in V.
$$
\end{defi}

Let $V$ be a vertex algebra, and fix $g$ to be an automorphism of $V$ of finite order $T$. Denote the imaginary unit by $\sqrt{-1}$. Then $V$ admits the following decomposition with respect to the action of $g$:
$$
V = \bigoplus_{r=0}^{T-1} V^r, \quad \text{where } V^r = \left\{ v \in V \mid gv = e^{-2\pi\sqrt{-1} r / T} v \right\}.
$$
In what follows, when we write $V^r$, we always assume $r \in \{0, 1, \ldots, T-1\}$. Note that $g$ commutes with $\mathcal{D}$, i.e., $g\mathcal{D} = \mathcal{D}g$, which implies $\mathcal{D}(V^r) \subset V^r$ for all $r$.

\begin{defi}
A conformal vector of a vertex algebra $V$ is a vector $\omega \in V$ such that the following conditions hold:

\begin{enumerate}
    \item For any $m, n \in \mathbb{Z}$,
    $$
    [L(m), L(n)] = (m - n)L(m + n) + \frac{m^3 - m}{12} \delta_{m+n, 0} c_V,
    $$
    where $L(n)=\omega_{n+1}
$
and $c_V \in \mathbb{C}$ is the central charge of $V$; and $L(-1) = \mathcal{D}$;

    \item The operator $L(0)$ is semisimple on $V$, i.e., 
    $$
    V = \bigoplus_{n \in \mathbb{C}} V_n, \quad \text{where } L(0)|_{V_n} = n \cdot \operatorname{id}_{V_n}.
    $$
    The eigenvalues of $L(0)$ are called the conformal weights. Moreover, the operators $L(n)$ for $n \geq 1$ act locally nilpotently on $V$.
\end{enumerate}

A vertex algebra $V$ equipped with a conformal vector $\omega$ is called a conformal vertex algebra. A vertex operator algebra is a conformal vertex algebra $V$ with integer conformal weights such that $\dim V_n < \infty$ for all $n \in \mathbb{Z}$, and $V_n = 0$ for sufficiently small $n$.
\end{defi}

Let $V$ be a conformal vertex algebra. For any $n \in \mathbb{C}$, elements in $V_n$ are said to be homogeneous. If $u \in V_n$, we define the weight of $u$ as $\operatorname{wt} u = n$. Whenever the notation $\operatorname{wt} u$ appears, it will always be understood that $u$ is a homogeneous vector.

An automorphism $g$ of $(V,\omega)$ is a vertex algebra automorphism preserving the conformal structure, i.e., $g(\omega) = \omega$. In particular, the definition of an automorphism of a vertex operator algebra coincides with that of an automorphism of the underlying conformal vertex algebra.

\begin{defi}
Let $V$ be a vertex operator algebra. A weak $g$-twisted $V$-module $M$ is a vector space equipped with a linear map
$$
Y_M(\cdot, z) \colon V \to (\operatorname{End} M)\left[\left[z^{1/T}, z^{-1/T}\right]\right],
$$
sending each $u \in V^r$ (for $0 \leq r \leq T - 1$) to 
$
Y_M(u, z) = \sum_{n \in r/T + \mathbb{Z}} u_n z^{-n - 1},
$
such that the following conditions hold:

\begin{enumerate}
    \item $Y_M(\mathbf{1}, z) = \operatorname{id}_M$;
    
    \item For any $u \in V^r$ and $w \in M$, we have $u_{r/T + n} w = 0$ for sufficiently large $n$;
    
    \item The twisted Jacobi identity: for any $u \in V^r$, $v \in V$,
    \begin{align*}
    &z_0^{-1} \delta\left(\frac{z_1 - z_2}{z_0}\right) Y_M(u, z_1) Y_M(v, z_2)
    - z_0^{-1} \delta\left(\frac{z_2 - z_1}{-z_0}\right) Y_M(v, z_2) Y_M(u, z_1) \\
    &\quad = z_2^{-1} \left(\frac{z_1 - z_0}{z_2}\right)^{-r/T} \delta\left(\frac{z_1 - z_0}{z_2}\right) Y_M(Y(u, z_0)v, z_2).
    \end{align*}
\end{enumerate}
\end{defi}

\begin{defi}
Let $V$ be a vertex operator algebra. An admissible $g$-twisted $V$-module is a $(1/T)\mathbb{N}$-graded weak $g$-twisted $V$-module
$
M = \bigoplus_{n \in (1/T)\N} M(n),
$
satisfying the condition
$$
v_m M(n) \subseteq M(n + \wt v - m - 1)
$$
for any  $v \in V$ and  $m, n \in (1/T)\mathbb{Z}$.
\end{defi}

\begin{defi}
Let $V$ be a vertex operator algebra. A (ordinary) $g$-twisted $V$-module is a $\mathbb{C}$-graded weak $g$-twisted $V$-module
$
M = \bigoplus_{\lambda \in \mathbb{C}} M_\lambda,
$
where
$$
M_\lambda = \{w \in M \mid L_M(0)w = \lambda w\},
$$
and $L_M(0)$ is the component operator of $Y_M(\omega, z) = \sum_{n \in \mathbb{Z}} L_M(n) z^{-n - 2}$. We further require that each homogeneous subspace $M_\lambda$ is finite-dimensional, and for each fixed $\lambda \in \mathbb{C}$, we have $M_{n/T + \lambda} = 0$ for all sufficiently negative integers $n$.
\end{defi}

\begin{defi}
A vertex operator algebra $V$ is called $g$-rational if the category of admissible $g$-twisted $V$-modules is semisimple. In particular, $V$ is called rational if it is $1$-rational.
\end{defi}

\begin{defi}
A vertex operator algebra $V$ is called $g$-regular if every weak $g$-twisted $V$-module $M$ is a direct sum of irreducible $g$-twisted $V$-modules.  $V$ is called regular if it is $1$-regular.
\end{defi}

Let $V$ be a vertex operator algebra, and let $g$ be an automorphism of $V$ of finite order $T$. For any $k, l \in \{0, 1, \ldots, T - 1\}$, define $\delta_k(l)$ as follows:
$$
\delta_k(l) =
\begin{cases}
1, & \text{if } k \geq l, \\
0, & \text{if } k < l,
\end{cases}
$$
and set $\delta_k(T) = 0$.
For any $n \in (1/T)\mathbb{N}$, there exists a unique $\bar{n} \in \{0, 1, \ldots, T - 1\}$ such that
$
n = \lfloor n \rfloor + {\bar{n}}/{T},
$
where $\lfloor \cdot \rfloor$ denotes the floor function.

For any $n \in (1/T)\mathbb{N}$, let $O_{g,n}(V)$ be the linear span of all elements of the form $u \circ_{g,n} v$ and $L(-1)u + L(0)u$, where for any $u \in V^r$ and $v \in V$, we define
$$
u \circ_{g,n} v = \operatorname{Res}_z  \frac{(1+z)^{\wt  u - 1 + \delta_{\bar{n}}(r) + \lfloor n \rfloor + r / T}}{z^{2\lfloor n \rfloor + \delta_{\bar{n}}(r) + \delta_{\bar{n}}(T - r) + 1}}Y(u, z)v .
$$
We also define a second product $*_{g,n}$ on $V$ as follows: for $u \in V^r$ and $v \in V$,
$$
u *_{g,n} v = \sum_{m=0}^{\lfloor n \rfloor} (-1)^m \binom{m + \lfloor n \rfloor}{\lfloor n \rfloor} \operatorname{Res}_z   \frac{(1+z)^{\wt  u + \lfloor n \rfloor}}{z^{\lfloor n \rfloor + m + 1}} Y(u, z)v,
$$
if $r = 0$, and $u *_{g,n} v = 0$ if $r > 0$.
Define $A_{g,n}(V)$ to be the quotient $V / O_{g,n}(V)$.
Let $M$ be any weak $g$-twisted $V$-module. Define the linear map $o(\cdot): V \to \operatorname{End} M$ by sending each homogeneous element $v \in V$ to $o(v) = v_{\wt  v - 1}$. For any $n \in (1/T)\mathbb{N}$, define
$$
\Omega_n(M) = \left\{ w \in M \mid v_{\wt  v - 1 + i} w = 0 \text{ for all } v \in V \text{ and all } i \in (1/T)\mathbb{Z} \text{ with } i > n \right\}.
$$

\begin{thm}\label{thm2.9} \cite{DLM3}
Let $V$ be a vertex operator algebra and let $M$ be a weak $g$-twisted $V$-module, with $n \in (1/T)\mathbb{N}$. Then the following statements hold:

\begin{enumerate}
    \item The subspace $O_{g,n}(V)$ is a two-sided ideal of $V$ with respect to the product $\circ_{g,n}$. Moreover, the quotient space $A_{g,n}(V)$ becomes an associative algebra under $*_{g,n}$, with identity element $\mathbf{1} + O_{g,n}(V)$. In addition, $\omega + O_{g,n}(V)$ lies in the center of $A_{g,n}(V)$.

    \item The identity map on $V$ induces a surjective homomorphism of associative algebras from $A_{g,n}(V)$ to $A_{g,n - 1/T}(V)$.

    \item The space $\Omega_n(M)$ is an $A_{g,n}(V)$-module, where the action of $u + O_{g,n}(V)$ is given by $o(u)$ for any $u \in V$.

    \item The vertex operator algebra $V$ is $g$-rational if and only if all $A_{g,n}(V)$ are finite-dimensional semisimple associative algebras.

    \item If $V$ is $g$-rational, then $V$ has only finitely many irreducible admissible $g$-twisted modules up to isomorphism, and every irreducible admissible $g$-twisted $V$-module is ordinary.
\end{enumerate}
\end{thm}
For a conformal vertex algebra $(V, \omega)$ of central charge $c$, there exists another conformal vertex algebra structure on $V$ defined by the vertex operator map
$$
Y[u, z] = Y\left(e^{z L(0)} u, e^z - 1\right),
$$
for all $u \in V$, with $\mathbf{1}$ as the vacuum vector and $\tilde{\omega} = \omega - \frac{1}{24} c \mathbf{1}$ as the new conformal vector (see \cite{Z1}). We denote this new conformal vertex algebra by $\exp(V, \omega)$ and write
$$
Y[\tilde{\omega}, z] = \sum_{n \in \mathbb{Z}} L[n] z^{-n-2}.
$$

Let $(V, \omega)$ be a conformal vertex algebra, $g$ an automorphism of $V$ of finite order $T$, and let $n \in (1/T)\mathbb{N}$. One can similarly define weak $g$-twisted $V$-modules for $(V, \omega)$. Furthermore, we can define the associative algebra $A_{g,n}(V, \omega)$ associated to the conformal vertex algebra $(V, \omega)$ under the assumption that $V$ has integer conformal weights.

Let $V$ be a conformal vertex algebra. Let $g_1, g_2, g_3$ be mutually commuting automorphisms of $V$ of finite order dividing a positive integer $T$, i.e., $g_k^T = 1$ for $k = 1, 2, 3$. For integers $j_1, j_2 \in \mathbb{Z}$, define
$$
V^{(j_1, j_2)} = \left\{ v \in V \mid g_k v = e^{-2\pi\sqrt{-1} j_k / T} v,\ k = 1, 2 \right\}.
$$
Then we have the decomposition
$
V = \bigoplus_{0 \leq j_1, j_2 < T} V^{(j_1, j_2)}.
$
The following definition is due to Xu (see \cite{X1}; cf. \cite{DL1}, \cite{FHL1}):

\begin{defi}
Let $V$ be a conformal vertex algebra. For $k = 1, 2, 3$, let $(M_k, Y_{M_k})$ be a weak $g_k$-twisted $V$-module. An intertwining operator of type $\binom{M_3}{M_1\ M_2}$ is a linear map
\begin{align*}
\mathcal{Y}(\cdot, x): &\ M_1 \to (\operatorname{Hom}(M_2, M_3))\{x\} \\
&\ w \mapsto \mathcal{Y}(w, x) = \sum_{h \in \mathbb{C}} w_h x^{-h - 1}
\end{align*}
satisfying the following conditions:

\begin{enumerate}
    \item For any $w^{(1)} \in M_1$, $w^{(2)} \in M_2$, and $h \in \mathbb{C}$,
    $$
    w^{(1)}_{h + n} w^{(2)} = 0 \quad \text{for } n \in \mathbb{Q} \text{ sufficiently large}.
    $$

    \item For any $v \in V^{(j_1, j_2)}$ with $j_1, j_2 \in \mathbb{Z}$, and $w^{(1)} \in M_1$, the following Jacobi identity holds on $M_2$:
    \begin{align*}
    & x_0^{-1} \left(\frac{x_1 - x_2}{x_0}\right)^{\frac{j_1}{T}} \delta\left(\frac{x_1 - x_2}{x_0}\right) Y_{M_3}(v, x_1) \mathcal{Y}(w^{(1)}, x_2) \\
    &\quad - e^{\frac{j_1}{T} \pi i} x_0^{-1} \left(\frac{x_2 - x_1}{x_0}\right)^{\frac{j_1}{T}} \delta\left(\frac{x_2 - x_1}{-x_0}\right) \mathcal{Y}(w^{(1)}, x_2) Y_{M_2}(v, x_1) \\
    &= x_2^{-1} \left(\frac{x_1 - x_0}{x_2}\right)^{-\frac{j_2}{T}} \delta\left(\frac{x_1 - x_0}{x_2}\right) \mathcal{Y}(Y_{M_1}(v, x_0) w^{(1)}, x_2).
    \end{align*}

    \item For any $w^{(1)} \in M_1$,
    $$
    \mathcal{Y}(L_{M_1}(-1) w^{(1)}, x) = \frac{d}{dx} \mathcal{Y}(w^{(1)}, x).
    $$
\end{enumerate}
We denote by $I_V\binom{M_3}{M_1\ M_2}$ the space of all such intertwining operators of the given type. The fusion rule is defined as
$$
N_{M_1, M_2}^{M_3} = \dim I_V\binom{M_3}{M_1\ M_2}.
$$
These numbers are commonly referred to as the fusion rules of the modules involved.
\end{defi}
\section{Associative algebras $\tilde{A}_{g, n}(V)  $}

Let $V$ be a vertex algebra, and let $g$ be an automorphism of $V$ of finite order $T$. For any $n \in (1/T)\mathbb{N}$, and for $u \in V^r$, $v \in V$, define the bilinear operations $\diamond_{g,n}$ and $\bullet_{g,n}$ as follows:
\begin{align*}
u \diamond_{g,n} v 
&= \operatorname{Res}_x \frac{e^{x(\delta_{\bar{n}}(r) + \lfloor n \rfloor + r / T)}}{(e^x - 1)^{2\lfloor n \rfloor + \delta_{\bar{n}}(r) + \delta_{\bar{n}}(T - r) + 1}} Y(u, x)v \\
&= \operatorname{Res}_y \frac{(1 + y)^{\delta_{\bar{n}}(r) + \lfloor n \rfloor + r / T - 1}}{y^{2\lfloor n \rfloor + \delta_{\bar{n}}(r) + \delta_{\bar{n}}(T - r) + 1}} Y(u, \log(1 + y))v,
\end{align*}
and
\begin{align*}
u \bullet_{g,n} v 
&= \sum_{m = 0}^{\lfloor n \rfloor} (-1)^m \binom{m + \lfloor n \rfloor}{\lfloor n \rfloor} \operatorname{Res}_x \frac{e^{x(\lfloor n \rfloor + 1)}}{(e^x - 1)^{\lfloor n \rfloor + m + 1}} Y(u, x)v \\
&= \sum_{m = 0}^{\lfloor n \rfloor} (-1)^m \binom{m + \lfloor n \rfloor}{\lfloor n \rfloor} \operatorname{Res}_y \frac{(1 + y)^{\lfloor n \rfloor}}{y^{\lfloor n \rfloor + m + 1}} Y(u, \log(1 + y))v,
\end{align*}
if $r = 0$, and $u \bullet_{g,n} v = 0$ if $r > 0$.
Let $\tilde{O}_{g,n}(V)$ be the subspace of $V$ spanned by all elements of the form $u \diamond_{g,n} v$ and $\mathcal{D} u$ for any $u, v \in V$. Define the quotient space
$
\tilde{A}_{g,n}(V) = V / \tilde{O}_{g,n}(V).
$
We will show in this section that $\tilde{A}_{g,n}(V)$, equipped with the multiplication induced by $\bullet_{g,n}$, has the structure of an associative algebra. Moreover, the image of the vacuum vector, $\mathbf{1} + \tilde{O}_{g,n}(V)$, serves as the identity element in this algebra.

From \cite{DLM1}, we have the following identities for any $s \in \mathbb{N}$:
\begin{align}
&\sum_{m=0}^s \binom{m + s}{s} \frac{(-1)^m (1 + z)^{s+1} - (-1)^s (1 + z)^m}{z^{s + m + 1}} = 1, \label{zuheeq1} \\
&\sum_{m=0}^s \binom{m + s}{s} (-1)^m \frac{m z + s + m + 1}{z^{s + m + 2}} = (-1)^s \binom{2s + 1}{s} \frac{s + 1}{z^{2s + 2}}, \label{zuheeq2} \\
&\sum_{m=0}^s (-1)^m \binom{m + s}{s} \left( \sum_{i=0}^{s - m} \binom{-m - s - 1}{i} (-1)^i \frac{z_2^i (1 + z_2)^m}{z_1^{i + m}} - \frac{1}{z_1^m} \right) = 0. \label{zuheeq3}
\end{align}

\begin{lem}\label{prop1}
For any $u \in V^r$, $v \in V$, and integers $m \geq k \geq0$, we have
$$
\operatorname{Res}_x \frac{e^{x(\delta_{\bar{n}}(r) + \lfloor n \rfloor + r/T + k)}}{(e^x - 1)^{2\lfloor n \rfloor + \delta_{\bar{n}}(r) + \delta_{\bar{n}}(T - r) + 1 + m}} Y(u, x)v \in \tilde{O}_{g,n}(V).
$$
\end{lem}

\begin{proof}
We begin by observing that
\begin{align*}
&\operatorname{Res}_x \frac{e^{x(\delta_{\bar{n}}(r) + \lfloor n \rfloor + r/T + k)}}{(e^x - 1)^{2\lfloor n \rfloor + \delta_{\bar{n}}(r) + \delta_{\bar{n}}(T - r) + 1 + m}} Y(u, x)v \\
=&\operatorname{Res}_y \frac{(1+y)^{\delta_{\bar{n}}(r) + \lfloor n \rfloor + r/T + k - 1}}{y^{2\lfloor n \rfloor + \delta_{\bar{n}}(r) + \delta_{\bar{n}}(T - r) + 1 + m}} Y(u, \log(1+y))v \\
=&\sum_{i=0}^k \binom{k}{i} \operatorname{Res}_y \frac{(1+y)^{\delta_{\bar{n}}(r) + \lfloor n \rfloor + r/T - 1}}{y^{2\lfloor n \rfloor + \delta_{\bar{n}}(r) + \delta_{\bar{n}}(T - r) + 1 + m - i}} Y(u, \log(1+y))v \\
=&\sum_{i=0}^k \binom{k}{i} \operatorname{Res}_x \frac{e^{x(\delta_{\bar{n}}(r) + \lfloor n \rfloor + r/T)}}{(e^x - 1)^{2\lfloor n \rfloor + \delta_{\bar{n}}(r) + \delta_{\bar{n}}(T - r) + 1 + m - i}} Y(u, x)v.
\end{align*}
Hence, it suffices to prove the lemma for the case $k = 0$ and $m \geq 0$. We proceed by induction on $m$.
The base case $m = 0$ is clear from the definition of $\tilde{O}_{g,n}(V)$.
Assume the result holds for all $m \leq l$. Consider $m = l + 1$. Using the inductive hypothesis and computing the residue involving $\mathcal{D}u$, we have:
\begin{align*}
& \operatorname{Res}_x \frac{e^{x(\delta_{\bar{n}}(r) + \lfloor n \rfloor + r/T)}}{(e^x - 1)^{2\lfloor n \rfloor + \delta_{\bar{n}}(r) + \delta_{\bar{n}}(T - r) + 1 + l}} Y(\mathcal{D}u, x)v \\
= & \operatorname{Res}_x \frac{e^{x(\delta_{\bar{n}}(r) + \lfloor n \rfloor + r/T)}}{(e^x - 1)^{2\lfloor n \rfloor + \delta_{\bar{n}}(r) + \delta_{\bar{n}}(T - r) + 1 + l}} \frac{d}{dx} Y(u, x)v \\
= & -\operatorname{Res}_x \left( \frac{d}{dx} \frac{e^{x(\delta_{\bar{n}}(r) + \lfloor n \rfloor + r/T)}}{(e^x - 1)^{2\lfloor n \rfloor + \delta_{\bar{n}}(r) + \delta_{\bar{n}}(T - r) + 1 + l}} \right) Y(u, x)v \\
= & -(\delta_{\bar{n}}(r) + \lfloor n \rfloor + r/T) \operatorname{Res}_x \frac{e^{x(\delta_{\bar{n}}(r) + \lfloor n \rfloor + r/T)}}{(e^x - 1)^{2\lfloor n \rfloor + \delta_{\bar{n}}(r) + \delta_{\bar{n}}(T - r) + 1 + l}} Y(u, x)v \\
& + (2\lfloor n \rfloor + \delta_{\bar{n}}(r) + \delta_{\bar{n}}(T - r) + 1 + l) \operatorname{Res}_x \frac{e^{x(\delta_{\bar{n}}(r) + \lfloor n \rfloor + r/T + 1)}}{(e^x - 1)^{2\lfloor n \rfloor + \delta_{\bar{n}}(r) + \delta_{\bar{n}}(T - r) + l + 2}} Y(u, x)v \\
= & (\lfloor n \rfloor + \delta_{\bar{n}}(T - r) + 1 - r/T + l) \operatorname{Res}_x \frac{e^{x(\delta_{\bar{n}}(r) + \lfloor n \rfloor + r/T)}}{(e^x - 1)^{2\lfloor n \rfloor + \delta_{\bar{n}}(r) + \delta_{\bar{n}}(T - r) + 1 + l}} Y(u, x)v \\
& + (2\lfloor n \rfloor + \delta_{\bar{n}}(r) + \delta_{\bar{n}}(T - r) + 1 + l) \operatorname{Res}_x \frac{e^{x(\delta_{\bar{n}}(r) + \lfloor n \rfloor + r/T)}}{(e^x - 1)^{2\lfloor n \rfloor + \delta_{\bar{n}}(r) + \delta_{\bar{n}}(T - r) + l + 2}} Y(u, x)v.
\end{align*}
By the induction hypothesis, the first term on the right-hand side of the last equality lies in $\tilde{O}_{g,n}(V)$. It follows that the second term on the right-hand side must also belong to $\tilde{O}_{g,n}(V)$, which completes the inductive step and proves the lemma.
\end{proof}

\begin{lem}\label{prop2}
The following statements hold:
\begin{enumerate}
    \item For any $u, v \in V$, we have
    $$
    Y(u, x)v \equiv Y(v, -x)u \mod \tilde{O}_{g,n}(V).
    $$

    \item If $r \neq 0$, then $V^r \subset \tilde{O}_{g,n}(V)$.

    \item For any $u \in V^0$ and $v \in V$, we have
    $$
    u \bullet_{g,n} v \equiv \sum_{m=0}^{\lfloor n \rfloor} (-1)^{\lfloor n \rfloor} \binom{m + \lfloor n \rfloor}{m} \operatorname{Res}_x \frac{e^{xm}}{(e^x - 1)^{\lfloor n \rfloor + m + 1}} Y(v, x)u \mod \tilde{O}_{g,n}(V).
    $$

    \item For any $u, v \in V^0$, the following congruence holds:
    $$
    u \bullet_{g,n} v - v \bullet_{g,n} u \equiv \operatorname{Res}_x Y(u, x)v \mod \tilde{O}_{g,n}(V).
    $$
\end{enumerate}
\end{lem}

\begin{proof}

(1) By definition, $\mathcal{D}u \in \tilde{O}_{g,n}(V)$. Using this together with \eqref{eq2.1}, we obtain
$$
Y(u, x)v = e^{x \mathcal{D}} Y(v, -x)u \equiv Y(v, -x)u \mod \tilde{O}_{g,n}(V),
$$
as required.

(2) Let $u \in V^r$ with $r \ne 0$. Then,
\begin{align*}
u \diamond_{g,n} \mathbf{1} 
&= \operatorname{Res}_y \frac{(1+y)^{\delta_{\bar{n}}(r) + \lfloor n \rfloor + r/T - 1}}{y^{2\lfloor n \rfloor + \delta_{\bar{n}}(r) + \delta_{\bar{n}}(T - r) + 1}} Y(u, \log(1+y))\mathbf{1} \\
&\equiv \operatorname{Res}_y \frac{(1+y)^{\delta_{\bar{n}}(r) + \lfloor n \rfloor + r/T - 1}}{y^{2\lfloor n \rfloor + \delta_{\bar{n}}(r) + \delta_{\bar{n}}(T - r) + 1}} u \mod \tilde{O}_{g,n}(V) \quad\quad\quad\quad\quad \text{(by part (1))} \\
&= \binom{\delta_{\bar{n}}(r) + \lfloor n \rfloor + r/T - 1}{2\lfloor n \rfloor + \delta_{\bar{n}}(r) + \delta_{\bar{n}}(T - r)} u.
\end{align*}
Since $r \ne 0$, the binomial coefficient is nonzero. As $u \diamond_{g,n} \mathbf{1} \in \tilde{O}_{g,n}(V)$, it follows that $u \in \tilde{O}_{g,n}(V)$, completing the proof of part (2).

(3) From part (1), we deduce
\begin{align*}
u \bullet_{g,n} v 
&\equiv \sum_{m=0}^{\lfloor n \rfloor} (-1)^m \binom{m + \lfloor n \rfloor}{m} \operatorname{Res}_x \frac{e^{x(\lfloor n \rfloor + 1)}}{(e^x - 1)^{\lfloor n \rfloor + m + 1}} Y(v, -x)u\bmod \tilde{O}_{g,n}(V) \\
&= \sum_{m=0}^{\lfloor n \rfloor} (-1)^{\lfloor n \rfloor} \binom{m + \lfloor n \rfloor}{m} \operatorname{Res}_x \frac{e^{xm}}{(e^x - 1)^{\lfloor n \rfloor + m + 1}} Y(v, x)u ,
\end{align*}
which proves part (3).

(4) Using part (3), we compute
\begin{align*}
& \quad u \bullet_{g,n} v - v \bullet_{g,n} u \\
&\equiv \operatorname{Res}_x \left( \sum_{m=0}^{\lfloor n \rfloor} \binom{m + \lfloor n \rfloor}{\lfloor n \rfloor} \frac{(-1)^m e^{x(\lfloor n \rfloor + 1)} - (-1)^{\lfloor n \rfloor} e^{xm}}{(e^x - 1)^{\lfloor n \rfloor + m + 1}} \right) Y(u, x)v \mod \tilde{O}_{g,n}(V) \\
&= \operatorname{Res}_y \sum_{m=0}^{\lfloor n \rfloor} \binom{m + \lfloor n \rfloor}{\lfloor n \rfloor} \frac{(-1)^m (1+y)^{\lfloor n \rfloor + 1} - (-1)^{\lfloor n \rfloor} (1+y)^m}{y^{\lfloor n \rfloor + m + 1}} Y\left((1+y)^{-1} u, \log(1+y)\right)v \\
&= \operatorname{Res}_y Y\left((1+y)^{-1} u, \log(1+y)\right)v \quad\quad\quad\quad\quad\quad\quad\quad\quad\quad\quad\quad\quad\quad\quad \text{(by \eqref{zuheeq1})} \\
&= \operatorname{Res}_x Y(u, x)v,
\end{align*}
which completes the proof of part (4).
\end{proof}
\begin{lem}\label{lem3.3}
The subspace $\tilde{O}_{g,n}(V)$ is a two-sided ideal of $V$ under the multiplication $\bullet_{g,n}$.
\end{lem}

\begin{proof}
To prove that $\tilde{O}_{g,n}(V)$ is a two-sided ideal, we need to show that for any $u, v \in V$,  $u \bullet_{g,n} v$ lies in $ \tilde{O}_{g,n}(V)$ whenever either $u$ or $v$ lies in $\tilde{O}_{g,n}(V)$. This follows from checking three  cases.

\textbf{Case (1):} We first show that $(\mathcal{D}u) \bullet_{g,n} v \in \tilde{O}_{g,n}(V)$ and $u \bullet_{g,n} (\mathcal{D}v) \in \tilde{O}_{g,n}(V)$ for all $u, v \in V$.
Without loss of generality, assume $u \in V^0$. Then, we have
\begin{align*}
&\quad(\mathcal{D}u) \bullet_{g,n} v \\
&= \sum_{m=0}^{\lfloor n \rfloor} (-1)^m \binom{m + \lfloor n \rfloor}{m} \operatorname{Res}_x \frac{e^{x(\lfloor n \rfloor + 1)}}{(e^x - 1)^{\lfloor n \rfloor + m + 1}} Y(\mathcal{D}u, x)v \\
&= \sum_{m=0}^{\lfloor n \rfloor} (-1)^m \binom{m + \lfloor n \rfloor}{m} \operatorname{Res}_x \frac{e^{x(\lfloor n \rfloor + 1)}}{(e^x - 1)^{\lfloor n \rfloor + m + 1}} \frac{d}{dx} Y(u, x)v \\
&= \sum_{m=0}^{\lfloor n \rfloor} (-1)^{m+1} \binom{m + \lfloor n \rfloor}{m} \operatorname{Res}_x \frac{e^{x(\lfloor n \rfloor + 1)} \left( (\lfloor n \rfloor + 1)(e^x - 1) - (\lfloor n \rfloor + m + 1)e^x \right)}{(e^x - 1)^{\lfloor n \rfloor + m + 2}} Y(u, x)v \\
&= \sum_{m=0}^{\lfloor n \rfloor} (-1)^{m+1} \binom{m + \lfloor n \rfloor}{m} \operatorname{Res}_y \frac{(my + \lfloor n \rfloor + m + 1)(1 + y)^{\lfloor n \rfloor}}{y^{\lfloor n \rfloor + m + 2}} Y(u, \log(1+y))v \\
&= (-1)^{\lfloor n \rfloor + 1} \binom{2\lfloor n \rfloor + 1}{\lfloor n \rfloor} (\lfloor n \rfloor + 1) \operatorname{Res}_y \frac{(1 + y)^{\lfloor n \rfloor}}{y^{2\lfloor n \rfloor + 2}} Y(u, \log(1+y))v \quad\quad\quad\quad\quad \text{(by \eqref{zuheeq2})} \\
&= (-1)^{\lfloor n \rfloor + 1} \binom{\lfloor n \rfloor + 1}{\lfloor n \rfloor} (\lfloor n \rfloor + 1) u \diamond_{g,n} v \in \tilde{O}_{g,n}(V).
\end{align*}
Next, by Lemma~\ref{prop2}(4), we have
$$
u \bullet_{g,n} (\mathcal{D}v) \equiv -\operatorname{Res}_x \frac{d}{dx} Y(v, x)u = 0 \mod \tilde{O}_{g,n}(V),
$$
which completes Case (1).

\textbf{Case (2):} Next, we show that $(u_1 \diamond_{g,n} u_2) \bullet_{g,n} u_3 \in \tilde{O}_{g,n}(V)$ for all $u_1, u_2, u_3 \in V$.
Assume without loss of generality that $u_1 \in V^{r_1}, u_2 \in V^{r_2}$ with $r_1 + r_2 \equiv 0 \mod T$. Define the functions:
$$
f(r) = \delta_{\bar{n}}(r) + \lfloor n \rfloor + r/T, \quad h(r) = 2\lfloor n \rfloor + \delta_{\bar{n}}(r) + \delta_{\bar{n}}(T - r) + 1.
$$
Note the identities:
\begin{equation}\label{zuheeq4}
   h(r_1) - f(r_1) = f(r_2), \quad h(r_1) = h(r_2). 
\end{equation}
Now compute:
\begin{align*}
    &(u_1\diamond_{g,n} u_2) \bullet_{g,n}u_3\\
    =&\sum_{m=0}^{\xqz{n}}(-1)^m\binom{m+\xqz{n}}{m}\Res_{x_2}\Res_{x_0}
\frac{e^{x_2(\xqz{n}+1)}}{(e^{x_2}-1)^{\xqz{n}+m+1}}\frac{e^{x_0f(r_1)}}{(e^{x_0}-1)^{h(r_1)}}Y(Y(u_1,x_0)u_2,x_2)u_3\\
    =&\sum_{m=0}^{\xqz{n}}(-1)^m\binom{m+\xqz{n}}{m}\Res_{x_1}\Res_{x_2}\Res_{x_0}
\frac{e^{x_2(\xqz{n}+1)}}{(e^{x_2}-1)^{\xqz{n}+m+1}}\frac{e^{x_0f(r_1)}}{(e^{x_0}-1)^{h(r_1)}}\\
    &\quad\cdot x_1^{-1}\delta\left(
    \frac{x_2+x_0}{x_1}\right) Y(Y(u_1,x_0)u_2,x_2)u_3\\
    =&\sum_{m=0}^{\xqz{n}}(-1)^m\binom{m+\xqz{n}}{m}\Res_{x_1}\Res_{x_2}\Res_{x_0}
\frac{e^{x_2(\xqz{n}+1)}}{(e^{x_2}-1)^{\xqz{n}+m+1}}\frac{e^{x_0f(r_1)}}{(e^{x_0}-1)^{h(r_1)}}\\
&\quad\cdot x_0^{-1} \delta\left(\frac{x_1-x_2}{x_0}\right) Y\left(u_1, x_1\right) Y\left( u_2, x_2\right) u_3\\
&-\sum_{m=0}^{\xqz{n}}(-1)^m\binom{m+\xqz{n}}{m}\Res_{x_1}\Res_{x_2}\Res_{x_0}
\frac{e^{x_2(\xqz{n}+1)}}{(e^{x_2}-1)^{\xqz{n}+m+1}}\frac{e^{x_0f(r_1)}}{(e^{x_0}-1)^{h(r_1)}}\\
&\quad\cdot x_0^{-1} \delta\left(\frac{x_2-x_1}{-x_0}\right) Y\left(u_2, x_2\right) Y\left(u_1, x_1\right) u_3\\
=&\sum_{m=0}^{\xqz{n}}(-1)^m\binom{m+\xqz{n}}{m}\Res_{x_1}\Res_{x_2}
\frac{e^{x_2(\xqz{n}+1)}}{(e^{x_2}-1)^{\xqz{n}+m+1}}\frac{e^{(x_1-x_2)f(r_1)}}{(e^{(x_1-x_2)}-1)^{h(r_1)}}\\
&\quad\cdot  Y\left( u_1, x_1\right) Y\left(u_2, x_2\right) u_3\\
&-\sum_{m=0}^{\xqz{n}}(-1)^m\binom{m+\xqz{n}}{m}\Res_{x_1}\Res_{x_2}
\frac{e^{x_2(\xqz{n}+1)}}{(e^{x_2}-1)^{\xqz{n}+m+1}}\frac{e^{(-x_2+x_1)f(r_1)}}{(e^{(-x_2+x_1)}-1)^{h(r_1)}}\\
&\quad\cdot  Y\left(u_2, x_2\right) Y\left( u_1, x_1\right) u_3\\
=&\sum_{m=0}^{\xqz{n}}(-1)^m\binom{m+\xqz{n}}{m}\Res_{x_1}\Res_{x_2}
\frac{e^{x_2(\xqz{n}+1)}}{(e^{x_2}-1)^{\xqz{n}+m+1}}\frac{e^{(x_1-x_2)f(r_1)}}{(e^{(x_1-x_2)}-e^{-x_2}+e^{-x_2}-1)^{h(r_1)}}\\
&\quad\cdot  Y\left( u_1, x_1\right) Y\left(u_2, x_2\right) u_3\\
&-\sum_{m=0}^{\xqz{n}}(-1)^m\binom{m+\xqz{n}}{m}\Res_{x_1}\Res_{x_2}
\frac{e^{x_2(\xqz{n}+1)}}{(e^{x_2}-1)^{\xqz{n}+m+1}}\frac{e^{(-x_2+x_1)f(r_1)}}{(e^{(-x_2+x_1)}-e^{x_1}+e^{x_1}-1)^{h(r_1)}} \\
&\quad\cdot Y\left(u_2, x_2\right) Y\left( u_1, x_1\right) u_3\\
=&\sum_{m=0}^{\xqz{n}}\sum_{k\geq0}(-1)^{m+k}\binom{m+\xqz{n}}{m}\binom{-h(r_1)}{k}\Res_{x_1}\Res_{x_2}\frac{e^{x_1f(r_1)}}{(e^{x_1}-1)^{h(r_1)+k}}\\
&\quad\cdot\frac{e^{x_2(\xqz{n}+1+h(r_1)-f(r_1))}}{(e^{x_2}-1)^{\xqz{n}+m+1-k}} Y\left(u_1, x_1\right) Y\left( u_2, x_2\right) u_3\\
&-\sum_{m=0}^{\xqz{n}}\sum_{k\geq0}(-1)^{m+h(r_1)+k}\binom{m+\xqz{n}}{m}\binom{-h(r_1)}{k}\Res_{x_1}\Res_{x_2}
\frac{e^{x_2(f(r_2)+\xqz{n}+k+1)}}{(e^{x_2}-1)^{h(r_2)+\xqz{n}+m+k+1}}\\
&\quad\cdot e^{x_1(f(r_1)-h(r_1)-k)}(e^{x_1}-1)^k
Y\left(u_2, x_2\right) Y\left( u_1, x_1\right) u_3\quad\quad\quad\quad\quad\quad\quad\quad\quad\quad\text{(by  \eqref{zuheeq4})}\\
\equiv&\,0\bmod \tilde{O}_{g,n}(V), \quad\quad\quad\quad\quad\quad\quad\quad\quad\quad\quad\quad\quad\quad\quad\quad\quad\quad\quad\quad\quad\quad\quad\text{(by Lemma~\ref{prop1})}
\end{align*}
which completes Case (2).

\textbf{Case (3):} Finally, we show that $u_3 \bullet_{g,n} (u_1 \diamond_{g,n} u_2) \in \tilde{O}_{g,n}(V)$ for all $u_1, u_2, u_3 \in V$.
Assuming $u_3 \in V^0$ and $u_1 \in V^{r_1}, u_2 \in V^{r_2}$ with $r_1 + r_2 \equiv 0 \mod T$ from the definition $\bullet_{g,n}$ and Lemma~\ref{prop2}(2). Then we have 
\begin{align*}
    &-u_3 \bullet_{g,n}(u_1\diamond_{g,n} u_2)\\
    \equiv& \operatorname{Res}_{x_2}\operatorname{Res}_{x_0}\frac{e^{x_0f(r_1)}}{(e^{x_0}-1)^{h(r_1)}} Y\left( Y(u_1,x_0)u_2, x_2\right) u_3  \bmod \tilde{O}_{g, n}(V)\quad\quad\quad\quad\text{(by Lemma~ \ref{prop2}(4))} \\
    = &\operatorname{Res}_{x_2}\operatorname{Res}_{x_0}\operatorname{Res}_{x_1}\frac{e^{x_0f(r_1)}}{(e^{x_0}-1)^{h(r_1)}} x_1^{-1}\delta\left(\frac{x_2+x_0}{x_1}\right)Y\left( Y(u_1,x_0)u_2, x_2\right) u_3\\
    =&\operatorname{Res}_{x_2}\operatorname{Res}_{x_0}\operatorname{Res}_{x_1}\frac{e^{x_0f(r_1)}}{(e^{x_0}-1)^{h(r_1)}} x_0^{-1}\delta\left(\frac{x_1-x_2}{x_0}\right)Y(u_1,x_1)Y\left( u_2, x_2\right) u_3\\
    &\quad-\operatorname{Res}_{x_2}\operatorname{Res}_{x_0}\operatorname{Res}_{x_1}\frac{e^{x_0f(r_1)}}{(e^{x_0}-1)^{h(r_1)}} x_0^{-1}\delta\left(\frac{-x_2+x_1}{x_0}\right)Y\left( u_2, x_2\right) Y(u_1,x_1)u_3\\
    =& \operatorname{Res}_{x_2}\operatorname{Res}_{x_1}\frac{e^{(x_1-x_2)f(r_1)}}{(e^{(x_1-x_2)}-1)^{h(r_1)}} Y\left(u_1, x_1\right) Y\left( u_2, x_2\right) u_3\\
    &\quad-\operatorname{Res}_{x_2}\operatorname{Res}_{x_1}\frac{e^{(x_1-x_2)f(r_1)}}{(e^{(-x_2+x_1)}-1)^{h(r_1)}} Y\left(u_2, x_2\right) Y\left( u_1, x_1\right) u_3\\
    =& \operatorname{Res}_{x_2}\operatorname{Res}_{x_1}\frac{e^{(x_1-x_2)f(r_1)}}{(e^{(x_1-x_2)}-e^{-x_2}+e^{-x_2}-1)^{h(r_1)}} Y\left(u_1, x_1\right) Y\left( u_2, x_2\right) u_3\\
    &\quad-\operatorname{Res}_{x_2}\operatorname{Res}_{x_1}\frac{e^{(x_1-x_2)f(r_1)}}{(e^{(-x_2+x_1)}-e^{x_1}+e^{x_1}-1)^{h(r_1)}} Y\left(u_2, x_2\right) Y\left( u_1, x_1\right) u_3\\
    =&\sum_{k\geq0}(-1)^k\binom{-h(r_1)}{k}\Res_{x_1}\Res_{x_2}\frac{e^{x_1f(r_1)}}{(e^{x_1}-1)^{h(r_1)+k}}e^{x_2(h(r_1)-f(r_1))}\\
    &\cdot(e^{x_2}-1)^k Y\left(u_1, x_1\right) Y\left( u_2, x_2\right) u_3\\
&-\sum_{k\geq0}(-1)^{h(r_1)+k}\binom{-h(r_1)}{k}\Res_{x_1}\Res_{x_2}
\frac{e^{x_2(f(r_2)+k)}}{(e^{x_2}-1)^{h(r_2)+k}}e^{x_1(f(r_1)-h(r_1)-k)}\\
&\cdot(e^{x_1}-1)^k Y\left(u_2, x_2\right) Y\left( u_1, x_1\right) u_3\quad\quad\quad\quad\quad\quad\quad\quad\quad\quad\quad\quad\quad\quad\quad\quad\quad\quad\text{(by \eqref{zuheeq4})}\\
\equiv&\,0\bmod \tilde{O}_{g,n}(V), \quad\quad\quad\quad\quad\quad\quad\quad\quad\quad\quad\quad\quad\quad\quad\quad\quad\quad\quad\quad\quad\quad\quad\text{(by  Lemma~\ref{prop1})}
\end{align*}
which completes Case (3).
Thus, in all cases, $\tilde{O}_{g,n}(V)$ is closed under left and right multiplication by arbitrary elements of $V$, proving that it is a two-sided ideal under $\bullet_{g,n}$.
\end{proof}

\begin{thm}\label{thm3.4}
The product $\bullet_{g,n}$ induces the structure of an associative algebra on the quotient space $\tilde{A}_{g,n}(V) = V / \tilde{O}_{g,n}(V)$, with identity element given by $\mathbf{1} + \tilde{O}_{g,n}(V)$.
\end{thm}

\begin{proof}
    We need to prove that $(u_1\bullet_{g,n}u_2)\bullet_{g,n}u_3\equiv u_1\bullet_{g,n}(u_2\bullet_{g,n}u_3)\bmod \tilde{O}_{g,n}(V)$ for any $u_1,u_2,u_3\in V$. By the definition of $\bullet_{g,n}$, Lemma~\ref{prop2}(2) and Lemma~\ref{lem3.3}, we may assume without loss of generality that $u_1,u_2 \in V^0$. Then we have
\begin{align*}
    &(u_1\bullet_{g,n}u_2)\bullet_{g,n}u_3\\
        =&\sum_{m_1,m_2=0}^{\xqz{n}}(-1)^{m_1+m_2}\binom{m_1+\xqz{n}}{\xqz{n}}\binom{m_2+\xqz{n}}{\xqz{n}}\\
        &\cdot\Res_{x_2}\Res_{x_0}\frac{e^{x_2(1+\xqz{n})}}{(e^{x_2}-1)^{\xqz{n}+m_2+1}}\frac{e^{x_0(1+\xqz{n})}}{(e^{x_0}-1)^{\xqz{n}+m_1+1}} Y(Y(u_1,x_0)u_2,x_2)u_3\\
        =&\sum_{m_1,m_2=0}^{\xqz{n}}(-1)^{m_1+m_2}\binom{m_1+\xqz{n}}{\xqz{n}}\binom{m_2+\xqz{n}}{\xqz{n}}\\
        &\cdot\Res_{x_2}\Res_{x_1}\frac{e^{x_2(1+\xqz{n})}}{(e^{x_2}-1)^{\xqz{n}+m_2+1}}\frac{e^{(x_1-x_2)(1+\xqz{n})}}{(e^{(x_1-x_2)}-1)^{\xqz{n}+m_1+1}}Y(u_1,x_1)Y(u_2,x_2)u_3\\
        &-\sum_{m_1,m_2=0}^{\xqz{n}}(-1)^{m_1+m_2}\binom{m_1+\xqz{n}}{\xqz{n}}\binom{m_2+\xqz{n}}{\xqz{n}}\\
        &\cdot\Res_{x_2}\Res_{x_1}\frac{e^{x_2(1+\xqz{n})}}{(e^{x_2}-1)^{\xqz{n}+m_2+1}}\frac{e^{(x_1-x_2)(1+\xqz{n})}}{(e^{(-x_2+x_1)}-1)^{\xqz{n}+m_1+1}} Y(u_2,x_2)Y(u_1,x_1)u_3\\
        =&\sum_{m_1,m_2=0}^{\xqz{n}}\sum_{k\geq 0}(-1)^{m_1+m_2+k}\binom{m_1+\xqz{n}}{\xqz{n}}\binom{m_2+\xqz{n}}{\xqz{n}}\binom{-\xqz{n}-m_1-1}{k}\\
        &\cdot\Res_{x_2}\Res_{x_1}\frac{e^{x_1(1+\xqz{n})}}{(e^{x_1}-1)^{\xqz{n}+m_1+1+k}}\frac{e^{x_2(1+\xqz{n}+m_1)}}{(e^{x_2}-1)^{\xqz{n}+m_2+1-k}}Y(u_1,x_1)Y(u_2,x_2)u_3\\
&-\sum_{m_1,m_2=0}^{\xqz{n}}\sum_{k\geq 0}(-1)^{m_2+\xqz{n}+k+1}\binom{m_1+\xqz{n}}{\xqz{n}}\binom{m_2+\xqz{n}}{\xqz{n}}\binom{-\xqz{n}-m_1-1}{k}\\
&\cdot\Res_{x_2}\Res_{x_1}\frac{e^{x_2(1+\xqz{n}+m_1+k)}}{(e^{x_2}-1)^{2\xqz{n}+m_2+m_1+k+2}}\frac{e^{x_1(-m_1-k)}}{(e^{x_1}-1)^{-k}}Y(u_2,x_2)Y(u_1,x_1)u_3\\
\equiv&\sum_{m_1,m_2=0}^{\xqz{n}}\sum_{k\geq 0}(-1)^{m_1+m_2+k}\binom{m_1+\xqz{n}}{\xqz{n}}\binom{m_2+\xqz{n}}{\xqz{n}}\binom{-\xqz{n}-m_1-1}{k}\\
        &\cdot\Res_{y_2}\Res_{y_1}\frac{(1+y_1)^{\xqz{n}}}{y_1^{\xqz{n}+m_1+1+k}}\frac{{(1+y_2)}^{\xqz{n}+m_1}}{y_2^{\xqz{n}+m_2+1-k}}\\
        &\cdot Y(u_1,\log(1+y_1))Y(u_2,\log(1+y_2))u_3\bmod \tilde{O}_{g,n}(V)\quad\quad\quad\quad\quad\quad\quad\text{(by Lemma~\ref{prop1})}\\
=&\sum_{m_1,m_2=0}^{\xqz{n}}\sum_{k= 0}^{\xqz{n}-m_1}(-1)^{m_1+m_2+k}\binom{m_1+\xqz{n}}{\xqz{n}}\binom{m_2+\xqz{n}}{\xqz{n}}\binom{-\xqz{n}-m_1-1}{k}\\
        &\cdot\Res_{y_2}\Res_{y_1}\frac{(1+y_1)^{\xqz{n}}}{y_1^{\xqz{n}+m_1+1+k}}\frac{{(1+y_2)}^{\xqz{n}+m_1}}{y_2^{\xqz{n}+m_2+1-k}} Y(u_1,\log(1+y_1))Y(u_2,\log(1+y_2))u_3\\
=&\sum_{m_1,m_2=0}^{\xqz{n}}(-1)^{m_1+m_2}\binom{m_1+\xqz{n}}{\xqz{n}}\binom{m_2+\xqz{n}}{\xqz{n}}\cdot\Res_{y_2}\Res_{y_1}\frac{(1+y_1)^{\xqz{n}}}{y_1^{\xqz{n}+m_1+1}}\frac{{(1+y_2)}^{\xqz{n}+m_1}}{y_2^{\xqz{n}+m_2+1}}\\
        &\quad\cdot Y(u_1,\log(1+y_1))Y(u_2,\log(1+y_2))u_3\quad\quad\quad\quad\quad\quad\quad\quad\quad\quad\quad\quad\quad\quad\text{(by \eqref{zuheeq3})}\\
        =&\,u_1\bullet_{g,n}(u_2\bullet_{g,n}u_3).    \end{align*}
        Next, we verify that $\mathbf{1} + \tilde{O}_{g,n}(V)$ acts as the identity element. For any $v \in V$, we compute:
    \begin{align*}
   \mathbf{1}\bullet_{g,n}v&=\sum_{m=0}^{\xqz{n}}(-1)^m\binom{m+\xqz{n}}{\xqz{n}}\Res_{x}\frac{e^{x(\xqz{n}+1)}}{(e^{x}-1)^{\xqz{n}+m+1}}Y(\mathbf{1},x)v \\
   &=\sum_{m=0}^{\xqz{n}}(-1)^m\binom{m+\xqz{n}}{\xqz{n}}\Res_{y}\frac{(1+y)^{\xqz{n}}}{y^{\xqz{n}+m+1}}Y(\mathbf{1},\log(1+y))v\\
   &=\sum_{m=0}^{\xqz{n}}(-1)^m\binom{m+\xqz{n}}{\xqz{n}}\Res_{y}\frac{(1+y)^{\xqz{n}}}{y^{\xqz{n}+m+1}}v=v.
\end{align*}
On the other hand, by Lemma~\ref{prop2}(4), we have:
$$
v \bullet_{g,n} \mathbf{1} \equiv \mathbf{1} \bullet_{g,n} v - \operatorname{Res}_x Y(\mathbf{1}, x)v = v \mod \tilde{O}_{g,n}(V).
$$
Therefore, $\mathbf{1} + \tilde{O}_{g,n}(V)$ is indeed the identity element in $\tilde{A}_{g,n}(V)$, completing the proof.
\end{proof}

\begin{rmk}
   For vertex operator algebra $V$ and $n \in \mathbb{N}$, the associative algebra $\tilde{A}_{n}(V)$ constructed in \cite{H3} is isomorphic to the associative algebra $\widetilde{A}_{\operatorname{id}_V,n}(V)$ in Theorem \ref{thm3.4} under $(2 \pi \sqrt{-1})^{L(0)}$.
\end{rmk}

\begin{prop}
The identity map on $V$ induces a surjective homomorphism of associative algebras from $\tilde{A}_{g,n}(V)$ to $\tilde{A}_{g,n - 1/T}(V)$.
\end{prop}

\begin{proof}
By Lemma~\ref{prop1}, we know that $\tilde{O}_{g,n}(V) \subset \tilde{O}_{g,n - 1/T}(V)$. Hence, the identity map on $V$ descends to a well-defined linear map between the quotient spaces. We only need to show that $u \bullet_{g,n} v \equiv u \bullet_{g,n-1/T}v \bmod \tilde{O}_{g,n-1/T}(V)$ for any $u, v \in V$. By the definition of $\bullet_{g,n}$ and $\bullet_{g,n-1/T}$, we may assume without loss of generality that $u \in V^0$, $n\in\Z$ and $\xqz{n-1/T}=n-1$. Then we have
\begin{align*}
& \quad u \bullet_{g,n} v\\
& =\sum_{m=0}^n(-1)^m\binom{m+n}{m} \operatorname{Res}_x \frac{e^{(n+1) x}}{\left(e^x-1\right)^{m+n+1}} Y(u, x) v \\
&=\sum_{m=0}^n(-1)^m\binom{m+n}{m} \operatorname{Res}_x \frac{e^{n x}}{\left(e^x-1\right)^{m+n}} Y(u, x) v \\
&\quad+\sum_{m=0}^n(-1)^m\binom{m+n}{m} \operatorname{Res}_x \frac{e^{n x}}{\left(e^x-1\right)^{m+n+1}} Y(u, x) v \\
&\equiv\sum_{m=0}^{n-1}(-1)^m\binom{m+n}{m} \operatorname{Res}_x \frac{e^{n x}}{\left(e^x-1\right)^{m+n}} Y(u, x) v \\
&\quad+\sum_{m=0}^{n-2}(-1)^m\binom{m+n}{m} \operatorname{Res}_x \frac{e^{n x}}{\left(e^x-1\right)^{m+n+1}} Y(u, x) v\bmod \tilde{O}_{g,n-1/T}(V) \\
&=\operatorname{Res}_x \frac{e^{n x}}{\left(e^x-1\right)^{n}} Y(u, x) v+\sum_{m=1}^{n-1}\operatorname{Res}_x \frac{e^{n x}}{\left(e^x-1\right)^{m+n}} Y(u, x) v\\
&\quad\cdot \left(
(-1)^m\binom{m+n}{m}+(-1)^{m+1}\binom{m+n-1}{m-1} 
\right)\\
&=u\bullet_{g,n-1/T}v.
\end{align*}
Thus, the identity map on $V$ induces a well-defined, surjective homomorphism of associative algebras from $\tilde{A}_{g,n}(V)$ to $\tilde{A}_{g,n - 1/T}(V)$, as claimed.
\end{proof}

\section{$g$-Rationality, $g$-regularity and twisted fusion rules}
In this section, we demonstrate that for any vertex operator algebra, the properties of $g$-rationality, $g$-regularity, and twisted fusion rules are independent of the choice of conformal vector.
\begin{lem}\label{lem4.1}
Let $(V, \omega)$ be a conformal vertex algebra with integer conformal weights, let $g$ be an automorphism of $V$ of order $T$, and let $n \in (1/T) \mathbb{N}$. Then we have
$$
A_{g,n}(V, \omega) = \tilde{A}_{g,n}(\exp(V, \omega)).
$$
\end{lem}

\begin{proof}
Note that $g$ is also an automorphism of $\exp(V, \omega)$.
For any $u \in V^0$ and $v \in V$, we compute:
\begin{align*}
u *_{g,n} v 
&= \sum_{m=0}^{\lfloor n \rfloor} (-1)^m \binom{\lfloor n \rfloor + m}{\lfloor n \rfloor} 
\operatorname{Res}_x \frac{(1+x)^{\lfloor n \rfloor}}{x^{\lfloor n \rfloor + m + 1}} 
Y\left((1+x)^{L(0)} u, x\right) v \\
&= \sum_{m=0}^{\lfloor n \rfloor} (-1)^m \binom{\lfloor n \rfloor + m}{\lfloor n \rfloor} 
\operatorname{Res}_z \frac{e^{(\lfloor n \rfloor + 1) z}}{(e^z - 1)^{\lfloor n \rfloor + m + 1}} 
Y\left(e^{z L(0)} u, e^z - 1\right) v \\
&= \sum_{m=0}^{\lfloor n \rfloor} (-1)^m \binom{\lfloor n \rfloor + m}{\lfloor n \rfloor} 
\operatorname{Res}_z \frac{e^{(\lfloor n \rfloor + 1) z}}{(e^z - 1)^{\lfloor n \rfloor + m + 1}} 
Y[u, z] v.
\end{align*}
Similarly, for any $u \in V^r$ and $v \in V$, we have:
\begin{align*}
u \circ_{g,n} v 
&= \operatorname{Res}_x \frac{(1+x)^{\delta_{\bar{n}}(r)+\lfloor n \rfloor + r/T - 1}}{x^{2\lfloor n \rfloor + \delta_{\bar{n}}(r) + \delta_{\bar{n}}(T - r) + 1}} 
Y\left((1+x)^{L(0)} u, x\right) v \\
&= \operatorname{Res}_z \frac{e^{(\delta_{\bar{n}}(r)+\lfloor n \rfloor + r/T) z}}{(e^z - 1)^{2\lfloor n \rfloor + \delta_{\bar{n}}(r) + \delta_{\bar{n}}(T - r) + 1}} 
Y\left(e^{z L(0)} u, e^z - 1\right) v \\
&= \operatorname{Res}_z \frac{e^{(\delta_{\bar{n}}(r)+\lfloor n \rfloor + r/T) z}}{(e^z - 1)^{2\lfloor n \rfloor + \delta_{\bar{n}}(r) + \delta_{\bar{n}}(T - r) + 1}} 
Y[u, z] v.
\end{align*}
From Theorem 4.2.1 of \cite{Z1}, we also know that $L[-1] = L(-1) + L(0)$. Therefore, by the definitions of $A_{g,n}(V, \omega)$ and $\tilde{A}_{g,n}(\exp(V, \omega))$, it follows that
$$
A_{g,n}(V, \omega) = \tilde{A}_{g,n}(\exp(V, \omega)).
$$
This completes the proof.
\end{proof}
\begin{thm}\label{thm4.2}
Let $(V, \omega)$ be a vertex operator algebra with an automorphism $g$ of finite order $T$, and let $n \in (1/T) \mathbb{N}$. Then the associative algebra $\tilde{A}_{g,n}(V, \omega)$ is isomorphic to the associative algebra $A_{g,n}(V, \omega)$.
\end{thm}

\begin{proof}
Let $B_j$ (for $j \in \mathbb{N}$) be rational numbers defined by the equation
$$
\log(1+y) = \left(\exp\left( \sum_{j \in \mathbb{N}} B_j y^{j+1} \frac{\partial}{\partial y} \right)\right) y.
$$
We denote the operator $\sum_{j \in \mathbb{N}} B_j L(j)$ by $L_+(B)$. By the change-of-variable formula in \cite{H1}, we have
$$
Y[u, z] = Y\left(e^{z L(0)} u, e^z - 1\right) = e^{-L_+(B)} Y\left(e^{L_+(B)} u, z\right) e^{L_+(B)}.
$$
This shows that $\exp(V, \omega)$ is isomorphic to $(V, \omega)$ as a vertex algebra. Consequently, $\tilde{A}_{g,n}(V, \omega)$ is isomorphic to $\tilde{A}_{g,n}(\exp(V, \omega))$.
By Lemma~\ref{lem4.1}, we have $A_{g,n}(V, \omega) = \tilde{A}_{g,n}(\exp(V, \omega))$. Therefore, it follows that $A_{g,n}(V, \omega)$ is isomorphic to $\tilde{A}_{g,n}(V, \omega)$, completing the proof.
\end{proof}

\begin{lem}\label{lem4.3}
Let $(V, \omega)$ and $(V, \omega')$ be two vertex operator algebra structures on the same underlying vertex algebra $V$, and let $g$ be an automorphism of both structures of order $T$. For any $n \in (1/T) \mathbb{N}$, the associative algebra $A_{g,n}(V, \omega)$ is isomorphic to $A_{g,n}(V, \omega')$.
\end{lem}

\begin{proof}
Since $(V, \omega)$ and $(V, \omega')$ are equal as vertex algebras, we have
$$
\tilde{A}_{g,n}(V, \omega) = \tilde{A}_{g,n}(V, \omega').
$$
Now, applying Theorem~\ref{thm4.2}, it follows that $A_{g,n}(V, \omega)$ is isomorphic to $A_{g,n}(V, \omega')$. This completes the proof.
\end{proof}

From Lemma~\ref{lem4.3} and Theorem~\ref{thm2.9}(4), we immediately obtain the following result.

\begin{thm}\label{thm4.4}
Let $(V, \omega)$ and $(V, \omega')$ be two vertex operator algebra structures on the same underlying vertex algebra $V$, and let $g$ be an automorphism of both structures of finite order $T$. Then $(V, \omega)$ is $g$-rational if and only if $(V, \omega')$ is $g$-rational. Moreover, the $g$-rationality of a vertex operator algebra is independent of the choice of the conformal vector.
\end{thm}

We now proceed to show that the $g$-regularity of a vertex operator algebra is also independent of the choice of the conformal vector.

\begin{thm}\label{thm4.5}
Let $(V, \omega)$ and $(V, \omega')$ be two vertex operator algebra structures on the same underlying vertex algebra $V$, and let $g$ be an automorphism of both structures of finite order $T$. Then $(V, \omega)$ is $g$-regular if and only if $(V, \omega')$ is $g$-regular. Moreover, the $g$-regularity of a vertex operator algebra is independent of the choice of the conformal vector.
\end{thm}

\begin{proof}
Assume that $(V, \omega)$ is $g$-regular. Then by Theorem~\ref{thm2.9}, every irreducible weak $g$-twisted $(V, \omega)$-module is an irreducible ordinary $g$-twisted $(V, \omega)$-module  and  $(V, \omega)$ is $g$-rational.

Furthermore, according to Theorem~\ref{thm2.9}, there are only finitely many irreducible  weak $g$-twisted $(V, \omega)$-modules up to isomorphism. Let $W_1, \ldots, W_r$ be a complete set of representatives of the isomorphism classes of irreducible  weak $g$-twisted $(V, \omega)$-modules. Since $(V, \omega)$ is $g$-regular, each $W_i$ is in fact an irreducible ordinary $g$-twisted $(V, \omega)$-module.

By Zhu's one-to-one correspondence for vertex operator algebras (see \cite{DLM2}), the associative algebra $A_g(V, \omega)$ has exactly $r$ irreducible modules up to isomorphism. Applying Lemma~\ref{lem4.3}, it follows that $A_g(V, \omega')$ also has exactly $r$ irreducible modules up to isomorphism.

Therefore, there exist exactly $r$ irreducible admissible $g$-twisted $(V, \omega')$-modules $M_1, \ldots, M_r$ up to isomorphism. By Theorems~\ref{thm2.9} and \ref{thm4.4}, these modules are also irreducible ordinary $g$-twisted $(V, \omega')$-modules. In particular, they are nonisomorphic irreducible weak $g$-twisted $(V, \omega')$-modules.

Moreover, since these modules are also irreducible weak $g$-twisted $(V, \omega)$-modules, they form a complete set of isomorphism-equivalence class representatives of irreducible weak $g$-twisted $(V, \omega)$-modules.

Hence, any weak $g$-twisted $(V, \omega)$-module is a direct sum of some of the $M_1, \ldots, M_r$, which implies that the same holds for weak $g$-twisted $(V, \omega')$-modules. Therefore, $(V, \omega')$ is $g$-regular, as desired.
\end{proof}
\begin{defi}
Let $(V, \omega)$ be a conformal vertex algebra with integer conformal weights and let $g$ be an automorphism of $(V, \omega)$ of order $T$. A weak $g$-twisted $V$-module $W$ is called a lowest weight weak $g$-twisted $V$-module if 
$
W = \bigoplus_{n \in (1/T)\mathbb{N}} W_{(\lambda+n)}
$
for some $\lambda \in \mathbb{C}$, where 
$$
W_{(\lambda+n)} = \{w \in W \mid L_W(0)w = (\lambda + n)w\}.
$$
Moreover, $W_{(\lambda)}$ is an irreducible $A_g(V, \omega)$-module and generates $W$ as a weak $g$-twisted $V$-module.
\end{defi}

Finally, we prove that the twisted fusion rules of a vertex operator algebra are also independent of the choice of the conformal vector.

\begin{thm}\label{thm4.6}
Let $(V, \omega)$ and $(V, \omega')$ be two vertex operator algebra structures on the same underlying vertex algebra $V$, and let $g_1, g_2, g_3$ be mutually commuting automorphisms of finite order of both structures. Let $(W_k, Y_{W_k})$ be a weak $g_k$-twisted $V$-module for $k = 2, 3$, and suppose $W_1$ is a lowest weight weak $g_1$-twisted $(V, \omega)$-module. Then there is a canonical isomorphism between the spaces of intertwining operators:
$$
I_{(V, \omega)}\binom{W_3}{W_1 \,\, W_2} \quad \text{and} \quad I_{(V, \omega')}\binom{W_3}{W_1 \,\, W_2}.
$$
\end{thm}

\begin{proof}
Let $a = \omega - \omega' \in V$. Then
$$
a_0 = \omega_0 - \omega'_0 = L(-1) - L'(-1) = \D-\D = 0 \quad \text{on } V.
$$
Let $\lambda$ denote the lowest $L(0)$-weight of $W_1$. Then $(W_1)_{(\lambda)}$ is an irreducible $A_{g_1}(V, \omega)$-module and generates $W_1$ as a weak $g_1$-twisted $V$-module.
From the identity
$$
[(a_{W_1})_0, Y_{W_1}(\omega, z)] = Y_{W_1}(a_0 \omega, z) = 0,
$$
we deduce that
$
[(a_{W_1})_0, L_{W_1}(0)] = 0,
$
which implies $(a_{W_1})_0 (W_1)_{(\lambda)} \subset (W_1)_{(\lambda)}$. Since $A_{g_1}(V, \omega)$ has countable dimension as a quotient of $V$, it follows that $(W_1)_{(\lambda)}$ also has countable dimension. Therefore, $(a_{W_1})_0$ acts on $(W_1)_{(\lambda)}$ as a scalar $\alpha \in \mathbb{C}$.
Since $(W_1)_{(\lambda)}$ generates $W_1$ and for any $v \in V$,
$$
[(a_{W_1})_0, Y_{W_1}(v, z)] = Y_{W_1}(a_0 v, z) = 0,
$$
it follows that $(a_{W_1})_0$ acts on all of $W_1$ as the scalar $\alpha$, i.e.,
$$
L_{W_1}(-1) - L'_{W_1}(-1) = \alpha \quad \text{on } W_1.
$$
Now, let $\mathcal{Y}(\cdot, z) \in I_{(V, \omega)}\binom{W_3}{W_1 \,\, W_2}$. Then for any $w \in W_1$, we have
$
\mathcal{Y}(L_{W_1}(-1)w, z) = \frac{d}{dz} \mathcal{Y}(w, z).
$
Therefore,
\begin{align*}
e^{-\alpha z} \mathcal{Y}(L'_{W_1}(-1)w, z) 
&= e^{-\alpha z} \mathcal{Y}(L_{W_1}(-1)w, z) - \alpha e^{-\alpha z} \mathcal{Y}(w, z) \\
&= e^{-\alpha z} \frac{d}{dz} \mathcal{Y}(w, z) - \alpha e^{-\alpha z} \mathcal{Y}(w, z) \\
&= \frac{d}{dz} \left(e^{-\alpha z} \mathcal{Y}(w, z)\right).
\end{align*}
This shows that $e^{-\alpha z} \mathcal{Y}(\cdot, z) \in I_{(V, \omega')}\binom{W_3}{W_1 \,\, W_2}$. Similarly, for any $\mathcal{Y}'(\cdot, z) \in I_{(V, \omega')}\binom{W_3}{W_1 \,\, W_2}$, we have $e^{\alpha z} \mathcal{Y}'(\cdot, z) \in I_{(V, \omega)}\binom{W_3}{W_1 \,\, W_2}$.
Hence, multiplication by $e^{\pm \alpha z}$ defines mutually inverse linear maps between the spaces of intertwining operators, yielding a canonical isomorphism:
$$
I_{(V, \omega)}\binom{W_3}{W_1 \,\, W_2} \cong I_{(V, \omega')}\binom{W_3}{W_1 \,\, W_2}.
$$
This completes the proof.
\end{proof}

\section*{Acknowledgment}
This work is supported by the National Natural Science Foundation of China (Grant No. 12271406). The author would like to express his sincere gratitude to his supervisor, Professor Jianzhi Han, for his constant guidance and support throughout this research. The author also wishes to thank Professor Haisheng Li for an insightful observation regarding vertex operator algebras—specifically, that the rationality of such an algebra is independent of the choice of its conformal vector. This valuable insight was shared during a short lecture series on vertex operator algebras organized by Professor Haibo Chen at Jimei University.

%\bibliographystyle{plain}  
%\bibliography{refs}

\end{document}